\documentclass{amsart}
\usepackage[shortalphabetic]{amsrefs}
\usepackage[utf8]{inputenc}
\usepackage{amsthm}
\usepackage[utf8]{inputenc}
\usepackage{appendix}
\usepackage{cite}
\usepackage{amssymb,amsmath,amsthm}
\usepackage{hyperref}

\newcommand{\amsprimary}[1]{{\footnotesize\noindent AMS 2010 \textit{Mathematics subject
classification:} Primary #1\vspace{1pc}}}
\newcommand{\keywordsnames}[1]{{\footnotesize\noindent\textit{Key words:} #1\vspace{1pc}}}

\newtheorem{theorem}{Theorem}
\newtheorem{teo}{Theorem}

\newtheorem{corollary}[teo]{Corollary}
\newtheorem{lemma}[teo]{Lemma}

\theoremstyle{definition}

\theoremstyle{remark}

\title[On March's criterion]{On March's criterion for transience on rotationally symmetric manifolds}
\author{Jhon E. Bravo and Jean C. Cortissoz}
\email{jhebravobu@unal.edu.co, jcortiss@uniandes.edu.co}
\address{Department of Mathematics, Universidad Nacional de Colombia, and
Department of Mathematics, Universidad de los Andes, Bogot\'a DC, Colombia}
\date{}

\begin{document}

\maketitle

\begin{abstract}
    In this short paper we show that March's criterion
    for the existence of a bounded non constant harmonic function
    on a weak model is also a necessary and sufficient condition for the solvability of
    the Dirichlet problem at infinity on a slight generalisation of a weak model (rotationally symmetric) metric
    on $\mathbb{R}^n$.
\end{abstract}

\keywordsnames{Dirichlet problem at infinity, bounded harmonic functions, model Riemannian space.}

{\amsprimary {31C05, 53C21}}

\section{Introduction}

March in \cite{March} gave the following criterion for the transience of a rotationally symmetric Riemannian manifold,
that is for a metric on
$\mathbb{R}^n$ of the form
\[
g=dr^2+\phi^2\left(r\right)g_{\mathbb{S}^{n-1}},
\]
with $\phi$ a smooth function such that $\phi>0$ for $r>0$, $\phi\left(0\right)=0$
and $\phi'\left(0\right)=1$, and where $g_{\mathbb{S}^{n-1}}$ is the standard round metric on $\mathbb{S}^{n-1}$. $M_g=\left(\mathbb{R}^n, g\right)$ is called
a weak model. It was proved by March in \cite{March} that $M_g$
supports bounded non constant harmonic functions if and only if
\begin{equation}
\label{ineq:March}
\int_{1}^{\infty} \phi^{n-3}\left(\sigma\right) \int_{\tau}^{\infty}\phi^{1-n}\left(\tau\right)\,d\tau\,d\sigma <\infty.
\end{equation}
The behaviour of the Brownian motion on a manifold and the existence of bounded non-trivial harmonic functions
are closely related (and to the interested reader we suggest the excellent survey \cite{Grigoryan}).

In this paper we go a bit beyond the result of March and show that for metrics of the form
\begin{equation}
\label{eq:general_metric}
g=dr^2+\phi^2\left(r\right)g_{\omega},
\end{equation}
where $g_{\omega}$ is any metric on $\mathbb{S}^{n-1}$, March's criterion (\ref{ineq:March}) implies the solvability of
the Dirichlet problem on $M$ at infinity. Thus we settle the problem
of establishing a necessary and sufficient condition on a weak model so that
the Dirichlet problem at infinity is solvable. Namely, we shall prove the following result.
\begin{theorem}
\label{thm:main}
Let $M=\left(\mathbb{R}^n, g\right)$  with $g$ of the form (\ref{eq:general_metric}).
For any continuous data $f\in C\left(\mathbb{S}^{n-1}\right)$,
the Dirichlet problem at infinity is uniquely solvable
on $M$ 
if and only if
\[
\int_{1}^{\infty} \phi^{n-3}\left(\sigma\right) \int_{\tau}^{\infty}\phi^{1-n}\left(\tau\right)\,d\tau\,d\sigma <\infty.
\]
\end{theorem}
It must be observed that the existence of bounded non-constant harmonic functions do not imply the
solvability of the Dirichlet problem at infinity, as Ancona shows in \cite{Ancona}.
Theorem \ref{thm:main} implies an almost sharp restriction on the curvature of a Cartan-Hadamard manifold
whose metric is of the form (\ref{eq:general_metric}) and where the Dirichlet problem at infinity is
solvable (see Theorem 2 in \cite{March}). Namely, following March's arguments, we have:
\begin{corollary}
Let $M=\left(\mathbb{R}^n,g\right)$ with a metric of the form (\ref{eq:general_metric}), and let $k\left(r\right)=-\phi''/\phi\left(r\right)$
be its radial curvature, $k\left(r\right)\leq 0$. Let $c_2=1$ and $c_n=\frac{1}{2}$ for $n\geq 3$. Then 
if $k\left(r\right)\leq -c/\left(r^2\log r\right)$ outside a compact set
for some $c>c_n$, the Dirichlet problem
at infinity is solvable. If 
$k\left(r\right)\geq -c/\left(r^2\log r\right)$ for some $c<c_n$ outside a compact set
 then the Dirichlet problem at infinity is not solvable.
\end{corollary}

The study of the Dirichlet problem at infinity has a history
of many deep and beautiful results (see 
for instance \cite{A, AS, Choi, Mil, Neel, Sullivan}).

\medskip
In \cite{Co2}, the concept of $L^2$-solvability for the Dirichlet problem at infinity is also introduced.
We also a have the following result.

\begin{theorem}
Let $M=\left(\mathbb{R}^n, g\right)$  with $g$ of the form (\ref{eq:general_metric}).
For any  $f\in L^2\left(\mathbb{S}^{n-1}\right)$,
 if
\[
\int_{1}^{\infty} \phi^{n-3}\left(\sigma\right) \int_{\tau}^{\infty}\phi^{1-n}\left(\tau\right)\,d\tau\,d\sigma <\infty.
\]
then the Dirichlet problem at infinity is  $L^2$-solvable
on $M$.
\end{theorem}
For the definition of $L^2$-solvability we refer the reader to Section \ref{sect:dirichlet_at_infinity} of this
paper and to \cite{Co2}.

\medskip
The layout of this paper is
as follows. In Section \ref{sect:dirichlet_at_infinity} we define what we understand by solving the Dirichlet problem at infinity;
in Section \ref{sect:proof_main} we prove our main result.

\section{Preliminaries}
\label{sect:dirichlet_at_infinity}

To define what we mean by the Dirichlet problem being solvable at infinity, we represent 
$\mathbb{R}^n$ as a cone over the sphere $\mathbb{S}^{n-1}$. That is
\[
\mathbb{R}^n \sim \left[0,\infty\right)\times \mathbb{S}^{n-1}/\left(\mathbb{S}^{n-1}\times\left\{0\right\}\right).
\]
We then can compactify $\mathbb{R}^n$ by defining
\[
\overline{\mathbb{R}^n}= \left[0,\infty\right]\times \mathbb{S}^{n-1}/\left(\mathbb{S}^{n-1}\times\left\{0\right\}\right),
\]
where $\left[0,\infty\right]$ is a compactification of $\left[0,\infty\right)$.

We say that the Dirichlet problem on 
$M=\left(\mathbb{R}^n,g\right)$ is solvable at infinity for boundary data $f\in C\left(\mathbb{S}^{n-1}_{\infty}\right)$,
if there is a harmonic function $u$ on $M$, that is $\Delta_g u= 0$ on $M$, which extends continuously to
a function $\overline{u}:\overline{\mathbb{R}^n}\longrightarrow \mathbb{R}$
and such that $\overline{u}\left(\infty,\cdot\right)=f\left(\cdot\right)$.

Our definition of solvability at infinity coincides with that of Choi \cite{Choi} on Cartan-Hadamard manifolds.
Indeed, 
for a Cartan-Hadamard manifold, 
the sphere at infinity given by the Eberlein-O'Neill compactification is obtained by 
adding at infinity the equivalence classes of 
geodesic rays starting at the pole
(a point around which the metric can be written as in (\ref{eq:general_metric})), where two geodesic rays are equivalent if 
$\limsup_{t\rightarrow \infty} d\left(\gamma_1\left(t\right),\gamma_2\left(t\right)\right)<\infty$,
and then endowing this set with cone topology (see \cite{Choi}); therefore,
if we identify the pole of the manifold 
with the vertex of the cone (that is, the equivalence
class of $\mathbb{S}^{n-1}\times \left\{0\right\}$), the set of
equivalence classes of geodesics can be identified with the unit sphere on $T_p M$, where $p$ is
a pole of the manifold,
which in turn implies
that it can be identified with $\mathbb{S}^{n-1}_{\infty}=\mathbb{S}^{n-1}\times \left\{\infty\right\}$.
From this our assertion follows.

\medskip
Next we recall the definition of $L^2$-solvability.
Given $f\in L^{2}\left(\mathbb{S}^{n-1}\right)$
we say that the Dirichlet problem is $L^2$-solvable at infinity with
boundary data $f\in L^{2}\left(\mathbb{S}^{n-1}\right)$ if there is a function $u:M\longrightarrow \mathbb{R}$
which is harmonic, and such that
\[
\lim_{r\rightarrow \infty} u\left(\omega, r\right)= f\left(\omega\right)
\quad 
\mbox{strongly in} \quad L^{2}\left(\mathbb{S}^{n-1}\right).\]

\section{Proof of the main result}
\label{sect:proof_main}

We shall concentrate on proving that March's
criterion implies the solvability of the Dirichlet
problem as the other direction is already proved in March's paper \cite{March}.
We will follow the approach implemented in \cite{Co2}, so we use separation of variables to find solutions to
the equation
\[
\Delta_g u=0.
\]
With this purpose in mind, we let
 $f_{m,k}$, $k=0,1,2,\dots, k_m$, be eigenfunctions of the $m$-th eigenvalue, $\lambda_m^2$,
 $m=0,1,2,\dots$, of the Laplacian $\Delta_{g_{\omega}}$ on $\left(\mathbb{S}^{n-1},g_{\omega}\right)$,
 and 
 such that the set $\left\{f_{m,k}\right\}_{m,k}$ is an orthonormal basis for $L^{2}\left(\mathbb{S}^{n-1}\right)$
 with respect to the inner product induced by the metric $g_{\omega}$. Then, if $\varphi_m$ is such that
 \[
 \Delta_g \left(\varphi_m f_{m,k}\right)=0,
 \]
 the equation that $\varphi_m$ must satisfy is the radial Laplacian equation
 \begin{equation}
 \label{eq:radial_laplacian}
 \varphi_m''+\left(n-1\right)\frac{\phi'}{\phi}\varphi_m'-\frac{\lambda_m^2}{\phi^2}\varphi_m=0.
 \end{equation}
It is proved in \cite{Co2} that
we may assume that $\varphi_0=1$, and that $\varphi_m$ for $m>0$ has the form
\[
\varphi_m\left(r\right)=r^l \alpha\left(r\right),
\]
where $l=\dfrac{-\left(n-2\right)+\sqrt{\left(n-2\right)^2+4\lambda_m^2}}{2}>0$ satisfies the indicial equation 
\[
l\left(l-1\right)+\left(n-1\right)l-\lambda_m^2=0,
\]
and $\alpha$ is a smooth function.
Also, it is shown that $\varphi_m$ can be chosen so that it is nondecreasing, and thus $\varphi_m\geq 0$. 
Therefore, as argued in \cite{Co2}, to show solvability of the Dirichlet
problem at infinity, it suffices to show that the $\varphi_m$ are bounded.

Indeed, let us sketch why this is so (for details we ask the reader
to consult \cite{Co2}) once we have proved that the $\varphi_m$'s are bounded we may without loss of
generality assume that $0\leq \varphi_m\leq 1$
and that $\lim_{r\rightarrow\infty}\varphi_m\left(r\right)=1$. Given $f\in C\left(\mathbb{S}^{n-1}_{\infty}\right)$
we can expand it in a Fourier series:
\[
\sum_{m,k} c_{m,k}f_{m,k}\left(\omega\right).
\]
If $f$ is smooth, a theorem of Peetre \cite{Peetre} guarantees that this series converges absolutely and uniformly
to $f$. Thus a harmonic function solving the Dirichlet problem at infinity with boundary data $f$ is given by
\[
u\left(r,\omega\right)=\sum_{m} \varphi_m\left(r\right) \sum_{k} c_{m,k}f_{m,k}\left(\omega\right).
\]
We call $u$ a \emph{harmonic extension} of $f$.
Next, if $f$ is just continuous, an approximation argument using smooth functions gives the
existence of a harmonic extension or $f$. Uniqueness follows from the Maximum Principle (see \cite{Co2}).

\medskip
In the case of $f\in L^2$, again, as proved in \cite{Co2}, the boundedness of $\varphi_m$ implies
solvability in this case.

\medskip
All we are left to show then is the boundedness of the $\varphi_m$'s, which do next.
\begin{lemma}
Given a solution $\varphi_{m}$ to the radial Laplacian (\ref{eq:radial_laplacian}), there are constants $A, B>0$ such that 
the following bound holds
\[
\varphi_{m}\left(s\right)\leq 
B\exp\left(\int_{1}^{s}\frac{\lambda_m^2}{\phi^{n-1}\left(\tau\right)}\left(A+\int_{1}^{\tau}\phi^{n-3}\left(\sigma\right)\, d\sigma
\right)\, d\tau\right),
\]
for all $s\geq 1$.
\end{lemma}

\begin{proof}
In order to proceed, we write
\[
\varphi_m\left(s\right) = B\exp\left(\int_1^s \frac{\lambda_m^2}{\phi^{n-1}\left(\tau\right)}x\left(\tau\right)\, d\tau\right),
\]
with $B=\varphi_m\left(1\right)$,
for $s\geq 1$, with $x$ a smooth function.  

From the equation satisfied by $\varphi_m$ we arrive at an equation for $x$:
\[
x'\left(s\right)+\frac{\lambda_m^2}{\phi^{n-1}\left(s\right)}x^2\left(s\right)=\phi^{n-3}\left(s\right),
\]
which leads us to the inequality
\[
x'\left(s\right)\leq \phi^{n-3}\left(s\right).
\]

This yields the estimate
\[
x\left(s\right)\leq A+ \int_{1}^{s} \phi^{n-3}\left(\sigma\right)\, d\sigma,
\]
for an appropriately chosen constant $A>0$, and the lemma follows.

\end{proof}

From the previous lemma it follows that if 
\[
\int_{1}^{\infty}\frac{1}{\phi^{n-1}\left(\tau\right)}\, d\tau+
\int_{1}^{\infty}\frac{1}{\phi^{n-1}\left(\tau\right)}\int_{1}^{\tau}\phi^{n-3}\left(\sigma\right)\, d\sigma \, d\tau < \infty,
\]
then $\varphi_m$ is bounded. Observe that the second integral can be rewritten as
\[
\int_{1}^{\infty} \phi^{n-3}\left(\sigma\right) \int_{\tau}^{\infty}\phi^{1-n}\left(\tau\right)\,d\tau\,d\sigma,
\]
and thus, since the inner integral must converge for almost all $\tau$, the convergence of the second integral implies the convergence of the integral
\[
\int_{1}^{\infty}\frac{1}{\phi^{n-1}\left(\tau\right)}\, d\tau,
\]
so we arrive at the fact that
\[
\int_{1}^{\infty} \phi^{n-3}\left(\sigma\right) \int_{\tau}^{\infty}\phi^{1-n}\left(\tau\right)\,d\tau\,d\sigma <\infty,
\]
implies that $\varphi_m$ is bounded. This finishes the proof of the results announced in the introduction of this 
paper.


\end{document}